\documentclass{amsart}
\usepackage[english]{babel}
\usepackage[utf8]{inputenc}
\usepackage[autostyle]{csquotes}

\usepackage{latexsym}
\usepackage{amsfonts}
\usepackage{euscript}
\usepackage{times}

\usepackage{dsfont}
\usepackage{amsmath}
\usepackage{enumitem}
\usepackage{amsthm, amssymb}
\usepackage{bm}
\usepackage{url}

\usepackage{booktabs}
\usepackage{caption}

\usepackage{imakeidx}
\makeindex[columns=1]

\theoremstyle{plain} 
\newtheorem{teo}{Theorem}[section]
\theoremstyle{plain} 
\newtheorem{teor}{Theorem}[section]

\theoremstyle{definition}
\newtheorem{defi}[teo]{Definition}
\theoremstyle{plain} 
\newtheorem{prop}[teo]{Proposition}
\theoremstyle{plain}
\newtheorem{lem}[teo]{Lemma}
\theoremstyle{plain}
	
\theoremstyle{definition}
\newtheorem{oss}[teo]{Remark}
\theoremstyle{definition}

\theoremstyle{definition}
\newtheorem{ese}[teo]{Example}
\theoremstyle{plain} 
\newtheorem{conj}{Conjecture}[section]

\newcommand*{\sn}{\unlhd \unlhd \ }

\begin{document}

\title{On the number of $p$-elements in a finite group}
\author{Pietro Gheri}
\address{Dipartimento di Matematica e Informatica ``U. Dini'',\newline
Universit\`a degli Studi di Firenze, viale Morgagni 67/a,
50134 Firenze, Italy.}
\email{pigheri@gmail.it}
		
\dedicatory{This work is dedicated to the memory of Carlo Casolo.\\
His knowledge, his curiosity, his humility and his humanity were an example to all of his students and friends.}

\begin{abstract}
In this paper we study the ratio between the number of $p$-elements and the order of a Sylow $p$-subgroup of a finite group $G$. As well known, this ratio is a positive integer and we conjecture that, for every group $G$, it is at least the $(1-\frac{1}{p})$-th power of the number of Sylow $p$-subgroups of $G$. We prove this conjecture if $G$ is $p$-solvable. Moreover, we prove that the conjecture is true in its generality if a somewhat similar condition holds for every almost simple group.
\end{abstract}		
		
\maketitle

\section{Introduction}

Let $G$ be a finite group and $p$ be a prime dividing the order of $G$. Moreover, let
\[
\mathfrak{U}_p(G) = \bigcup_{P \in Syl_p(G)} P,
\]
be the set of $p$-elements of $G$.

A celebrated theorem of F.G.~Frobenius (\cite{frobenius:sylow}) states that if $P$ is a Sylow $p$-subgroup of $G$, then $|P|$ divides $|\mathfrak{U}_p(G)|$. We will call the positive integer $|\mathfrak{U}_p(G)|/|P|$ the \emph{$p$-Frobenius ratio of $G$}. 

The number of $p$-elements of a finite group is a fundamental invariant in finite group theory. Several different proofs of Frobenius' theorem have been given (see, for example, \cite{isaacs:frobenius} and \cite{speyer:frobenius}). Moreover in \cite[Theorem 15.2]{steinberg:endom} it is proven that the $p$-Frobenius ratio in a finite group of Lie type is equal to the size of a Sylow $p$-subgroup.

Nevertheless, it is still unknown if the Frobenius ratio has a combinatorial meaning. 

It is clear that the $p$-Frobenius ratio is $1$ if and only if $G$ contains a normal Sylow $p$-subgroup. In \cite{miller:theoryapplications}, it is proven with a nice and easy argument that if the $p$-Frobenius ratio is not $1$, then it must be greater or equal than $p$.

In this paper, we focus on the search for ``good'' bounds for the $p$-Frobenius ratio in terms of the number $n_p(G)$ of Sylow $p$-subgroups of $G$.

Of course, a trivial upper bound is obtained when every pair of Sylow p-subgroups of G has trivial intersection, so that, given a Sylow $p$-subgroup $P$ of $G$,
\[
\frac{|\mathfrak{U}_p(G)|}{|P|} \leq n_p(G)- \frac{n_p(G)-1}{|P|} \leq n_p(G).
\]

It is not hard to find examples of sequences of groups that show that a lower bound on the $p$-Frobenius ratio cannot be linear in $n_p(G)$. We state the following conjecture.

\begin{conj}
Let $G$ be a finite group, $p$ be a prime dividing $|G|$ and $P$ a Sylow $p$-subgroup of $G$. Then
\begin{equation} \label{conj pfrobratio}
\frac{|\mathfrak{U}_p(G)|}{|P|} \geq n_p(G)^{1-\frac{1}{p}}.
\end{equation}
\end{conj}

We will show in Example \ref{tight frob ratio} that this bound is ``asymptotically tight''. We show that Conjecture \ref{conj pfrobratio} is true for $p$-solvable groups. Namely, we prove the following.

\begin{teor} \label{bound on Omega psolv}
Let $G$ be a finite $p$-solvable group and $P$ be a Sylow $p$-subgroup of $G$. If $n_p(G)$ denotes the number of Sylow $p$-subgroups in $G$, then
\begin{equation} \label{bound p el}
\frac{|\mathfrak{U}_p(G)|}{|P|} \geq n_p(G)^{\frac{p-1}{p}}.
\end{equation}
\end{teor}

Inspired by the proof of Theorem \ref{bound on Omega psolv}, we show that a sufficient condition for Conjecture \ref{conj pfrobratio} to be true in general is that
\begin{equation} \label{conj lambda}
\left( \prod_{x \in P} \lambda_G(x) \right)^{1/|P|} \leq n_p(G)^{\frac{1}{p}},
\end{equation}
where for every $p$-element $x$ of $P$, $\lambda_G(x)$ denotes the number of Sylow $p$-subgroups of $G$ containing $x$.

For this condition we give a reduction to almost simple groups.

\begin{teor} \label{red alm simp}
Inequality (\ref{conj lambda}) holds for every finite group if and only if it holds for every finite almost simple group.
\end{teor}

One of the most important tools used here is the so-called Wielandt's subnormalizer (see Definition \ref{subnormalizer}), which is related to the number of $p$-elements (see Lemma \ref{form Omega subnor}). This connection is mainly due to the works of C.~Casolo on subnormalizers (\cite{casolo:subnor}, \cite{casolo:subnorsolv}). 

Another fundamental tool for the proof of our result is a theorem by G.~Navarro and N.~Rizo, concerning the number of fixed points in a coprime action of a $p$-group.

Throughout the paper $G$ will be a finite group and $p$ a prime dividing $|G|$. Also, for all $x \in G$ we denote with $x^G$ the conjugacy class of $x$ in $G$. 

\section{The p-solvable case}

In this section we prove Theorem \ref{bound on Omega psolv}. 
First of all we introduce the we introduce the concept of subnormalizer, whose definition (see \cite[pag. 238]{lennox:subnormal}) is inspired by the celebrated Wielandt's subnormality criterion, which says that a subgroup $H$ of $G$ is subnormal in $G$ if and only if $H$ is subnormal in $\langle H, g \rangle$ for every $g \in G$. 

\begin{defi} \label{subnormalizer} Let $H$ be a subgroup of $G$. The \textit{subnormalizer of $H$ in $G$}\index{Subnormalizer} is the set
\[
S_G(H) = \lbrace g \in G \ | \ H \sn \langle H,g \rangle \rbrace.
\]
where $\sn$ means ``is subnormal in''.
\end{defi}

A useful link between subnormalizers and the number of p-elements in a finite group (see Lemma \ref{form Omega subnor}) is established by using a beautiful theorem by C.~Casolo In order to state this theorem we introduce some notation. Let $H$ be a $p$-subgroup of $G$ and $P$ be a Sylow $p$-subgroup of $G$. We write $\lambda_G(H)$ for the number of Sylow $p$-subgroups of $G$ containing $H$ and $\alpha_G(H)$ for the number of $G$-conjugates of $H$ contained in $P$ (note that this number does not depend on the Sylow subgroup $P$ we are considering). 
When $H = \langle x \rangle $ is a cyclic subgroup, we simply write $S_G(x)$ and $\lambda_G(x)$, in place of $S_G(\langle x \rangle)$ and $\lambda_G(\langle x \rangle)$. In a similar fashion, we write $\alpha_G(x)$ for the number of $G$-conjugates of the element $x$ contained in $P$. We thus have that 
\begin{equation} \label{alpha x}
\alpha_G(x)=\alpha_G \left( \langle x \rangle \right) |N_G(\langle x \rangle)|/|C_G(x)|.
\end{equation} 

We can now state the aforementioned theorem by C.~Casolo.

\begin{teo}[\cite{casolo:subnor},\cite{casolo:subnorsolv}] \label{form subnor theorem}  Let $H$ be a $p$-subgroup of $G$. Then the following holds and $P$ be a Sylow $p$-subgroup of $G$.
\begin{itemize}
\item[a)] 
\begin{equation*} 
|S_G(H)|=\lambda_G(H) |N_G(P)| = \alpha_G(H) |N_G(H)|.
\end{equation*}
\item[b)] If $G$ is $p$-solvable and $\mathcal{M}$ is the set of all $p'$-factors in a given normal $\lbrace p,p' \rbrace$-series of $G$, then
\[
|S_G(H)|= |P| \prod_{U/V \in \mathcal{M}} \left| C_{U/V} \left( HV/V \right) \right|.
\]
\end{itemize}
\end{teo}

A first easy application of this result is a formula that expresses the number of $p$-elements in $G$ in terms of the orders of the subnormalizers of the cyclic subgroups of a Sylow $p$-subgroup of $G$.

\begin{lem} \label{form Omega subnor} Let $P$ be a Sylow $p$-subgroup of $G$. We have
\[
|\mathfrak{U}_p(G)|=\sum_{x \in P} \frac{|G|}{|S_G(x)|}.
\]
\end{lem}
\begin{proof}
In the sum
\[
\sum_{x \in P} |x^G| 
\] 
every class of $p$-elements is involved and its contribution is repeated as many times as the cardinality $|x^G \cap P|=\alpha_G(x)$. Hence
\[
|\mathfrak{U}_p(G)| = \sum_{x \in P} \frac{|x^G|}{\alpha_G(x)}=\sum_{x \in P} \frac{|G|}{\alpha_G(x)|C_G(x)|}=\sum_{x \in P} \frac{|G|}{|S_G(x)|},
\]
by part \textit{a)} of Theorem \ref{form subnor theorem} and formula (\ref{alpha x}).
\end{proof}

We now turn to the proof of Theorem \ref{bound on Omega psolv}. 
Another fundamental tool that we are going to use in the proof is the following formula proved by Navarro and Rizo.

\begin{teo}[\cite{navarro:brauwiel}] \label{nav riz} Suppose that $P$ is a $p$-group acting on a $p'$-group $G$. Then
\[
|C_G(P)| = \left( \prod_{x \in P} \frac{|C_G(x)|}{|C_G(x^p)|^{1/p}} \right)^{\frac{p}{(p-1)|P|}}.
\]
\end{teo}

We can now prove Theorem \ref{bound on Omega psolv}. 

\begin{proof}[Proof of Theorem \ref{bound on Omega psolv}]
By Lemma \ref{form Omega subnor} we have that the $p$-Frobenius ratio of $G$ is the arithmetic mean of the ratios 
\[
\frac{|G|}{|S_G(x)|}
\] 
when $x$ runs across $P$. By the Arithmetic-Geometric Mean Inequality, we get
\begin{equation} \label{FrobRatio arit geom}
\frac{|\mathfrak{U}_p(G)|}{|P|} = \frac{1}{|P|} \left( \sum_{x \in P} \frac{|G|}{|S_G(x)|} \right) \geq \left( \prod_{x \in P} \frac{|G|}{|S_G(x)|} \right)^{1/|P|}.
\end{equation}
Since $G$ is $p$-solvable we can take a normal $\lbrace p,p' \rbrace$-series, whose set of $p'$-factors we call $\mathcal{M}$. Then, by part \textit{b)} of Theorem \ref{form subnor theorem}, we have for all $x \in P$
\[
\frac{|G|}{|S_G(x)|} = \frac{|G|/|P|}{\prod_{U/V \in \mathcal{M}}|C_{U/V}(Vx)|} = \prod_{U/V \in \mathcal{M}} \frac{|U/V|}{|C_{U/V}(Vx)|}.
\]
We insert this last term in (\ref{FrobRatio arit geom}) and swap the products to get
\begin{eqnarray*}
\frac{|\mathfrak{U}_p(G)|}{|P|} & \geq & \left( \prod_{U/V \in \mathcal{M}} \left( \prod_{x \in P}  \frac{|U/V|}{|C_{U/V}(xV)|} \right) \right)^{1/|P|} \\
& = & \left( \prod_{U/V \in \mathcal{M}} \left| \frac{U}{V} \right|^{|P|} \left( \prod_{x \in P}  \frac{1}{|C_{U/V}(xV)|} \right) \right)^{1/|P|}.
\end{eqnarray*}
Now for all $U/V \in \mathcal{M}$, $P$ is a $p$-group that acts on the $p'$-group $U/V$. We can then apply Theorem \ref{nav riz} and use the trivial inequality $|C_{U/V}((xV)^p)| \leq |U/V|$, so that we have
\begin{eqnarray*}
\prod_{x \in P}  \frac{1}{|C_{U/V}(xV)|} & = & \left( \prod_{x \in P} \frac{1}{|C_{U/V}((xV)^p)|^{1/p}} \right) \frac{1}{|C_{U/V}(P)|^{|P|(p-1)/p}} \\
& \geq & \left( \frac{1}{|U/V|^{|P|/p}}\right)\frac{1}{|C_{U/V}(P)|^{|P|(p-1)/p}}
\end{eqnarray*}
and so
\[
\frac{|\mathfrak{U}_p(G)|}{|P|} \geq \left( \prod_{U/V \in \mathcal{M}} \frac{|U/V|}{|C_{U/V}(P)|} \right)^{\frac{p-1}{p}} = \left( \frac{|G|}{|S_G(P)|} \right)^{1-\frac{1}{p}},
\]
again by part \textit{b)} of Theorem \ref{form subnor theorem}. 

Finally, we observe that for a Sylow $p$-subgroup $S_G(P)=N_G(P)$, so that
\[
\frac{|\mathfrak{U}_p(G)|}{|P|}  \geq \left( \frac{|G|}{|S_G(P)|} \right)^{1-\frac{1}{p}} = \left( \frac{|G|}{|N_G(P)|} \right)^{1-\frac{1}{p}} = n_p(G)^{1-\frac{1}{p}}.
\]
\end{proof}

It is worth mentioning that the bound in Theorem \ref{bound on Omega psolv} is asymptotically tight in the sense specified by the following example.

\begin{ese} \label{tight frob ratio}
Let $p$ be a prime and, for $n$ a positive integer, let $P$ be an elementary abelian group of order $p^n$. Moreover set $\mathcal{M}$ to be the set of the maximal subgroups of $P$. Choose a prime $q$ such that $p$ divides $q-1$. Then for any $M \in \mathcal{M}$ we have that $P/M \simeq C_p$ acts fixed point freely as a group of automorphisms on a cyclic group $\langle a_M \rangle \simeq C_q$. We denote the image of the generator $a_M$ under this action by $a_M^{xM}$,  for every $xM \in P/M$.

Since $\bigcap_{M \in \mathcal{M}} M =1$, it follows that $P$ acts faithfully on the direct product $N$ of the groups $\langle a_M \rangle$. To be more explicit the following map
\begin{align*}
P & \rightarrow Aut(N) \\ 
x & \mapsto \phi_x,
\end{align*}
where $\phi_x(a_M)=a_M^{xM}$, for all $M \in \mathcal{M}$ is an injective homomorphism.

We consider the semidirect product $G_n = N \rtimes P$. The normalizer of $P$ in $G_n$ is $C_N(P)P=P$, hence the number of Sylow $p$-subgroups of $G_n$ is 
\begin{equation}\label{npgn}
n_p(G_n)= |N| = q^{|\mathcal{M}|}=q^{\frac{p^n-1}{p-1}}.
\end{equation}
In order to count the number of $p$-elements in $G_n$ we use the equality
\[
| \mathfrak{U}_p(G_n) | = \sum_{x \in P}  \frac{n_p(G_n)}{\lambda_{G_n}(x)},
\]
which follows from Lemma \ref{form Omega subnor} and part \textit{a)} of Theorem \ref{form subnor theorem}.
We thus have to compute $\lambda_{G_n}(x)$, for $x \in P \setminus \lbrace 1 \rbrace$. Using again part \textit{a)} of Theorem \ref{form subnor theorem}, we have
\[
\lambda_{G_n}(x)= \frac{\alpha_{G_n}(x)n_p(G_n)}{|x^{G_n}|}.
\]
Now since $P$ is abelian and $G_n$ has a normal $p$-complement, we have $\alpha_{G_n}(x)=1$ and $|x^{G_n}|=|N|/|C_N(x)|$, so that $\lambda_{G_n}(x)=|C_N(x)|$.
Given $x \in P \setminus \lbrace 1 \rbrace$, the centralizer of $x$ in $N$ is generated by those $a_M$ such that $a_M^{xM}=a_M$.
Since $P/M$ acts fixed point freely on $\langle a_M \rangle$, this holds if and only if $x \in M$, hence
\[
C_N(x)= \langle a_M \ | \ x \in M \rangle.
\] 
The number of maximal subgroups in $P$ containing a fixed nontrivial element is $\frac{p^{n-1}-1}{p-1}$, and so
\[
|C_N(x)|=q^{\frac{p^{n-1}-1}{p-1}}.
\]
We can then calculate the $p$-Frobenius ratio of $G_n$
\begin{equation}
\begin{split}
\frac{|\mathfrak{U}_p(G_n)|}{|P|} &= \frac{1}{|P|} \sum_{x \in P} \frac{n_p(G_n)}{\lambda_{G_n}(x)}  \\
& = \frac{1}{p^n} \left( 1 + (p^n-1) \frac{q^\frac{p^{n}-1}{p-1}}{q^\frac{p^{n-1}-1}{p-1}} \right) \\
& = \frac{1}{p^n} + \frac{p^n-1}{p^n} q^{p^{n-1}}.
\end{split}
\end{equation}
By (\ref{npgn}) we have
\[
n_p(G_n)^{\frac{p-1}{p}}= q^{\frac{p^n-1}{p}}.
\]
We can now compare the two members of the inequality stated by Theorem \ref{bound on Omega psolv}. By considering the limit
\[
\lim_{n \rightarrow \infty} \frac{|\mathfrak{U}_p(G_n)|/|P|}{n_p(G_n)^{\frac{p-1}{p}}} = \lim_{n \rightarrow \infty} \left( \frac{1}{p^n q^{\frac{p^n-1}{p}}} + \frac{p^n-1}{p^n} \frac{q^{p^{n-1}}}{q^{\frac{p^n-1}{p}}} \right) = q^{1/p},
\]
we see that the $p$-Frobenius ratio of $G_n$ and the $\left( 1-\frac{1}{p} \right)$-th power of the number of Sylow $p$-subgroups have the same asymptotic behaviour.
\end{ese}

\section{The general case}

In this section we explain why inequality (\ref{conj lambda}) is sufficient for establishing Conjecture \ref{conj pfrobratio} and we prove Theorem \ref{red alm simp}. 

Let $P$ a Sylow $p$-subgroup of $G$. Since Lemma \ref{form Omega subnor} is true for any group, by applying the Arithmetic-Geometric Mean Inequality as in \ref{FrobRatio arit geom}, we get 
\[
\frac{|\mathfrak{U}_p(G)|}{|P|} \geq \left( \prod_{x \in P} \frac{|G|}{|S_G(x)|} \right)^{1/|P|}
\]
and, recalling Theorem \ref{form subnor theorem}, we can write
\[
\frac{|\mathfrak{U}_p(G)|}{|P|} \geq \left( \prod_{x \in P} \frac{n_p(G)}{\lambda_G(x)} \right)^{1/|P|}.
\]
A sufficient condition for (\ref{bound p el}) is then
\[
\left( \prod_{x \in P} \frac{n_p(G)}{\lambda_G(x)} \right)^{1/|P|} \geq n_p(G)^{\frac{p-1}{p}},
\]
that is
\begin{equation*} 
\left( \prod_{x \in P} \lambda_G(x) \right)^{1/|P|} \leq n_p(G)^{\frac{1}{p}},
\end{equation*}
which is inequality (\ref{conj lambda}).

\begin{oss} 
In \cite{gheri:degnil} it is proven that if $x$ is a $p$-element of $G$ which is not contained in the $O_p(G)$, then $\lambda_G(x)$ is at most $n_p(G)/(p+1)$.
Focusing on a single element, this is the best one can get.
Inequality (\ref{conj lambda}), if true, would give a better bound on average, as it states that the geometric mean of the number of Sylow $p$-subgroups containing an element of a Sylow $p$-subgroup is at most the $p$-th root of the total number of Sylow $p$-subgroups.
\end{oss}

\begin{oss}
The bound (\ref{conj lambda}), if true, is best possible in a strict sense. If we compute the terms of inequality (\ref{conj lambda}) for the groups $G_n$ defined in Example \ref{tight frob ratio}, an equality occurs.
\end{oss}

\begin{oss} \label{Op 1} For the proof of Theorem \ref{red alm simp}, we can assume $O_p(G)=1$. This is because if $N$ is a normal $p$-subgroup of $G$, then, for all $x \in \mathfrak{U}_p(G)$, we have that $\lambda_G(x)=\lambda_{G/N}(xN)$, and so (\ref{conj lambda}) holds for $G$ if and only if it holds for $G/N$.
\end{oss}

First of all, we can reduce (\ref{conj lambda}) to nonsolvable groups all of whose proper quotients are solvable (see Proposition \ref{reduction conj mns}). We need some technical lemmas, the first of whom is proved in \cite[Lemma 3.3]{gheri:unsesto}.

\begin{lem}\label{sp x cresce sui sottogruppi}
Let $H$ be a subgroup of $G$ and $x \in H$ be a $p$-element. Then
\[
\frac{\lambda_G(x)}{n_p(G)} \leq \frac{\lambda_H(x)}{n_p(H)}.
\]
Moreover, if $H \unlhd G$, then the equality holds.
\end{lem}

\begin{lem} \label{SNx}
Let $N$ be a normal subgroup of $G$, $P$ be a Sylow $p$-subgroup of $G$ and $x$ an element of $P$. Assume that $G=NP$. Then
\[
|S_G(x)|=|S_N(x)| |NP/N|
\]
\end{lem}
\begin{proof}
If $x \in N$, the thesis follows from the fact that $|S_G(x)|/|G|=|S_N(x)|/|N|$, which can be easily derived from part \textit{a)} of Theorem \ref{form subnor theorem}.

We work by induction on $m$, where $|NP/N|=p^m$. If $m=0$, then $G=N$ and there is nothing to prove. Suppose that $m > 0$ and $G \neq N \langle x \rangle$. Let $M$ be a maximal subgroup of $G$ containing $N \langle x \rangle$. Then, as $M \unlhd G$, by inductive hypothesis 
\[
|S_G(x)|=|S_M(x)|p=|S_N(x)|p^{m-1}p.
\]
Finally, if $G=N \langle x \rangle$, we observe that, given $a \in N$ and $t$ a positive integer, $ax^t \in S_G(x)$ if and only if $\langle x \rangle$ is subnormal in $\langle x, ax^t \rangle = \langle x, a \rangle $, that is if and only if $a \in S_G(x)$. It follows that 
\[
|S_G(x)|=|S_N(x)| |NP/N|.
\]
\end{proof}

\begin{lem} \label{lambda np}
Let $N$ be a normal subgroup of $G$, $P$ be a Sylow $p$-subgroup of $G$ and $x \in P$. Then
\[
\lambda_G(x)= \lambda_{\frac{G}{N}}(Nx) \lambda_{NP}(x).
\]
\end{lem} 
\begin{proof}
First of all we show that given a $p$-element $x$, the value $\lambda_{NP}(x)$ is independent of the particular Sylow $p$-subgroup $P$ containing $x$. By Theorem \ref{form subnor theorem} and Lemma \ref{SNx}, we have
\[
\lambda_{NP}(x)= \frac{|S_{NP}(x)|}{|N_{NP}(P)|} =\frac{|S_{N}(x)|}{|N_{NP}(P)|} \left| \frac{NP}{N} \right| = \frac{|S_{N}(x)|}{|N|} |n_p(NP)| .
\]
If $Q$ is another Sylow $p$-subgroup such that $x \in Q$, then of course $n_p(NP)=n_p(NQ)$, since $NP$ and $NQ$ are conjugated in $G$. Moreover $|S_N(x)|$ depends only on $N$ and $x$.

Now let $\Delta_G^x$ be the set of the Sylow $p$-subgroups of $G$ containing $x$. We define the map
\begin{equation}
\begin{split}
\Delta_G^x & \rightarrow \Delta_{\frac{G}{N}}^{xN} \\
Q & \mapsto NQ/N.
\end{split}
\end{equation}
For all $\tilde{Q} \in \Delta_G^x$, the fiber of $N\tilde{Q}/N \in  \Delta_{\frac{G}{N}}^{xN}$ is the set of Sylow $p$-subgroups $Q$ of $G$ containing $x$ and such that $NQ=N\tilde{Q}$, that is, $\Delta_{N \tilde{Q}}^{x}$.
Since we proved that $\lambda_{N\tilde{Q}}(x)$ is independent of $\tilde{Q}$, we have
\[
\lambda_G(x)= | \Delta_{G}^{x}| = \left| \Delta_{\frac{G}{N}}^{xN}\right| \left| \Delta_{NP}^{x} \right| =  \lambda_{\frac{G}{N}}(xN) \lambda_{NP}(x).
\]
\end{proof}

\begin{prop} \label{reduction conj mns}
A counterexample of minimal order to inequality (\ref{conj lambda}) is a nonsolvable group having a unique minimal normal subgroup $M$, which is nonsolvable, and such that $G=MP$, where $P$ is a Sylow $p$-subgroup of $G$.
\end{prop}
\begin{proof}
Let $G$ be a counterexample of minimal order to inequality \ref{conj lambda} and let $P$ be a Sylow $p$-subgroup of $G$. By Remark \ref{Op 1} we have that $O_p(G)=1$. By Theorem \ref{bound on Omega psolv}, $G$ is nonsolvable. We show that every proper quotient of $G$ is solvable. Let $M$ be a minimal normal subgroup of $G$. By Lemma \ref{lambda np}, we have
\begin{eqnarray*}
\prod_{x \in P} \lambda_G(x) & = & \prod_{x \in P} \lambda_{\frac{G}{M}}(Mx) \lambda_{MP}(x) \\
&=& \left( \prod_{xM \in \frac{PM}{M}} \lambda_{\frac{G}{M}}(Mx) \right)^{|P \cap M|} \left( \prod_{x \in P} \lambda_{MP}(x) \right). 	
\end{eqnarray*}
By the minimality of $G$, we have
\[
\prod_{xM \in \frac{PM}{M}} \lambda_{\frac{G}{M}}(Mx) \leq n_p \left( G/M \right)^{\frac{|PM/M|}{p}}.
\]
If $MP<G$, we can again assume that the inequality is true for $MP$ and so
\[
\prod_{x \in P} \lambda_G(x) \leq n_p\left( \frac{G}{M} \right)^{\frac{|PM/M|}{p}|P \cap M|} n_p(MP)^\frac{|P|}{p}=n_p(G)^\frac{|P|}{p}.
\]
Since this is true for every minimal normal subgroup of $G$, the usual subdirect product argument gives that $G$ has a unique minimal normal subgroup $M$, which is nonsolvable and such that $G=MP$.
\end{proof}

Referring to the notation of Lemma \ref{reduction conj mns}, $M$ is the direct product of simple groups permuted by $P$. The next easy lemma loosely bounds the number of $\langle x \rangle$-invariant Sylow $p$-subgroups of $M$, where $x \in P$, in terms of the action of $\langle x \rangle$ on the direct factors of $M$. 

\begin{lem} \label{num orb}
Let $M \unlhd G$ be a direct product of $k$ copies of a group $L$, $M=L_1 \times \dots \times L_k$, let $x \in G$ be an element that permutes the factors $L_i$ of $M$ and let $p$ be a prime number. Then the number of $\langle x \rangle$-invariant Sylow $p$-subgroups of $M$ is at most $n_p(L)^s$, where $s$ is the number of orbits of $\langle x \rangle$ on $\Delta= \lbrace L_1, \dots, L_k \rbrace$.
\end{lem}
\begin{proof}
A Sylow $p$-subgroup $Q$ of $M$ is the direct product of $k$ Sylow $p$-subgroups of $L$, $Q = Q_1 \times \dots \times Q_k$. Suppose that $Q$ is normalized by $x$. If $L_i=L_j^{x^r}$, for $r \in \mathbb{Z}$, then $Q_i=Q_j^{x^r}$ and so one has at most $n_p(L)$ choices for each $\langle x \rangle$-orbit in $\Delta$. 
\end{proof}

We can now prove Theorem \ref{red alm simp}

\begin{proof}[Proof of Theorem \ref{red alm simp}]
Suppose that inequality (\ref{conj lambda}) is true for all finite almost simple groups and let $G$ be a counterexemple of minimal order. By Proposition \ref{reduction conj mns}, $G=MP$ where $P$ is a Sylow $p$-subgroup of $G$ and $M$ is the unique minimal normal subgroup
\[
M = L_1 \times \dots \times L_k, \ L_i \simeq L, \ \forall i \in \lbrace 1, \dots , k \rbrace.
\]
for some nonabelian simple group $L$. Moreover $P$ acts transitively on the set $\Delta = \lbrace L_1, \dots , L_k \rbrace$.
Since we are assuming the result for almost simple groups, we have $k>1$. 

Let $Q=P \cap M$. For any subgroup $Q \leq X \leq P$ we set $m_X$ to be the ratio 
\begin{equation} \label{mX}
m_X=\frac{|N_M(X)|}{|N_M(P)|}= \frac{n_p(G)}{n_p(MX)}.
\end{equation}
The last equality holds since $n_p(G)=[MP:N_{MP}(P)]=[M:N_M(P)]$ and $n_p(MX)=[MX:N_{MX}(X)]=[M:N_M(X)]$. Moreover observe that if $g \in N_M(P)$ and $x\in X$ then
\[
x^g=x[x,g] \in X(M \cap P)=XQ=X,
\]
and so $N_M(P) \leq N_M(X)$. 

With a slight abuse of notation, we denote with $\lambda_M(x)$ the number of Sylow $p$-subgroups in $M$ normalized by $x$ even for $x \notin M$. It is easy to check that $\lambda_M(x)=\lambda_{M\langle x \rangle}(x)$.

Let $H_0=N_P(L_1)$ be the stabilizer of $L_1$ in the action of $P$ on $\Delta$. Since $P$ is transitive on $\Delta$, $H_0 \neq P$. Choose now a maximal subgroup $H$ of $P$ containing $H_0$. Since $H$ is normal in $P$ and the stabilizers of the subgroups $L_i$ are all conjugated in $P$, we have that $H$ contains all of them. It follows that every element $x \in P \setminus H$ has at most $k/p$ orbits on $\Delta$ and so by Lemma \ref{num orb} 
\begin{equation}\label{orbits}
\lambda_M(x) \leq n_p(L)^{\frac{k}{p}}.
\end{equation}

We now consider separately the elements inside and outside $H$. As for the elements inside $H$, since $MH$ is normal in $G$, using Lemma \ref{sp x cresce sui sottogruppi}, we get
\begin{equation*}
\begin{split}
\prod_{x \in H} \lambda_G(x) &= \prod_{x \in H} \frac{n_p(G)}{n_p(MH)}\lambda_{MH}(x) \\
&= \prod_{x \in H} m_H \lambda_{MH}(x)=(m_H)^{|H|} \prod_{x \in H} \lambda_{MH}(x). 
\end{split}
\end{equation*}
Since $H$ is a Sylow $p$-subgroup of $MH$ and inequality (\ref{conj lambda}) holds for $MH<G$ we get
\begin{equation} \label{in H}
\begin{split}
\prod_{x \in H} \lambda_G(x) & = (m_H)^{|H|} \prod_{x \in H} \lambda_{MH}(x) \\
& \leq (m_H)^{|H|} n_p(MH)^\frac{|H|}{p}=m_H^{\frac{p-1}{p}|H|} n_p(G)^\frac{|H|}{p},
\end{split}
\end{equation}
where we applied (\ref{mX}).

Now we turn our attention on elements in $P \setminus H$. Let $\mathcal{T}$ be a set of representatives for the right cosets of $Q$ in $P$ that are not contained in $H$. The cardinality of $\mathcal{T}$ is then 
\[
|\mathcal{T}|=[P:Q]-[H:Q]=\frac{|P|-|H|}{|Q|}=(p-1) \frac{|H|}{|Q|}.
\]
We have, by Lemma \ref{sp x cresce sui sottogruppi}
\begin{eqnarray*}
\prod_{x \in P \setminus H} \lambda_G(x) & = & \prod_{x \in \mathcal{T}} \prod_{g \in Q} \lambda_G(gx) \leq \prod_{x \in \mathcal{T}} \left( \prod_{g \in Q} m_{Q\langle gx \rangle} \lambda_{M\langle gx \rangle}(gx) \right) \\
&=& \prod_{x \in \mathcal{T}} \left( m_{Q\langle x \rangle}^{|Q|} \prod_{g \in Q} \lambda_{M}(gx) \right).
\end{eqnarray*}
Now the elements $gx$ in the previous product are not in $H$ and so by (\ref{orbits})
\begin{equation} \label{prod lambda outside H}
\begin{split}
\prod_{x \in P \setminus H} \lambda_G(x) & \leq \left( \prod_{x \in \mathcal{T}} m_{Q\langle x \rangle} \right)^{|Q|} n_p(L)^{\frac{k}{p}|Q||\mathcal{T}|}\\ 
& = \left( \prod_{x \in \mathcal{T}} m_{Q\langle x \rangle} \right)^{|Q|} n_p(M)^{\frac{p-1}{p}|H|}.
\end{split}
\end{equation}

We now want to evaluate the product $\prod_{x \in \mathcal{T}} m_{Q\langle x \rangle}$. In the following we use the bar notation for the quotients modulo $Q$. If $R=N_M(Q)$ we have a coprime action of $\bar{P}$ on the $p'$-group $\bar{R}$. We apply Theorem \ref{nav riz} to this action and get
\begin{equation*}
|C_{\bar{R}}(\bar{P})|^{|\bar{P}|\frac{p-1}{p}} = \prod_{x \in \bar{P}} \frac{|C_{\bar{R}}(x)|}{|C_{\bar{R}}(x^p)|^{1/p}}.
\end{equation*}
Separating the elements inside $\bar{H}$ and those outside $\bar{H}$ and applying twice Theorem \ref{nav riz}, we get
\begin{equation*}
\begin{split}
|C_{\bar{R}}(\bar{P})|^{|\bar{P}|\frac{p-1}{p}} & = \left( \prod_{x \in \bar{H}} \frac{|C_{\bar{R}}(x)|}{|C_{\bar{R}}(x^p)|^{1/p}} \right) \left( \prod_{x \in \bar{P} \setminus \bar{H}} \frac{|C_{\bar{R}}(x)|}{|C_{\bar{R}}(x^p)|^{1/p}} \right) \\
& = |C_{\bar{R}}(\bar{H})|^{|\bar{H}|\frac{p-1}{p}} \left( \prod_{x \in \bar{P} \setminus \bar{H}} \frac{|C_{\bar{R}}(x)|}{|C_{\bar{R}}(x^p)|^{1/p}} \right).
\end{split}
\end{equation*}
Using the bound $|C_{\bar{R}}(x^p)| \leq |\bar{R}|$ and the fact that $|\bar{P} \setminus \bar{H}|=(p-1) \left| \bar{H}\right|$, we have
\begin{equation} \label{prodCRx}
\begin{split}
\prod_{x \in \bar{P} \setminus \bar{H}} |C_{\bar{R}}(x)| & = |C_{\bar{R}}(\bar{P})^{|\bar{P}|\frac{p-1}{p}} |C_{\bar{R}}(\bar{H})|^{-|\bar{H}|\frac{p-1}{p}} \left( \prod_{x \in \bar{P} \setminus \bar{H}} |C_{\bar{R}}(x^p)|^{1/p} \right) \\
 &\leq |C_{\bar{R}}(\bar{P})|^{|\bar{P}|\frac{p-1}{p}} |C_{\bar{R}}(\bar{H})|^{-|\bar{H}|\frac{p-1}{p}} |\bar{R}|^{|\bar{H}|\frac{p-1}{p}}.
\end{split} 
\end{equation}
We now observe that if $X$ is a subgroup of $P$ containing $Q$, then $Q=X \cap M$, so that $N_M(X) \leq N_M(Q)=R$ and we have
\begin{equation} \label{NNX CRX}
\overline{N_M(X)}=C_{\bar{R}}(\bar{X}).
\end{equation}
and since $Q \leq N_M(X)$ 
\[
m_X=\frac{|N_M(X)|}{|N_M(P)|} = \frac{\left| \overline{N_M(X)} \right|}{\left| \overline{N_M(P)} \right|} = \frac{|C_{\bar{R}}(\bar{X})|}{|C_{\bar{R}}(\bar{P})|}
\]
Hence, for all $x \in \mathcal{T}$,
\[
m_{Q\langle x \rangle} = \frac{\left| C_{\bar{R}}(xQ) \right|}{\left| C_{\bar{R}}(\bar{P}) \right|}.
\]
Going back to our product and using (\ref{prodCRx}), we then get
\begin{equation*} 
\begin{split}
\prod_{x \in \mathcal{T}} m_{Q\langle x \rangle} & = \prod_{x \in \mathcal{T}} \frac{|C_{\bar{R}}(xQ)|}{|C_{\bar{R}}(\bar{P})|} =\prod_{x \in \bar{P} \setminus \bar{H}} \frac{|C_{\bar{R}}(x)|}{|C_{\bar{R}}(\bar{P})|} \\
& \leq |C_{\bar{R}}(\bar{P})|^{-|\bar{P}|\frac{p-1}{p}} |C_{\bar{R}}(\bar{P})|^{|\bar{P}|\frac{p-1}{p}} |C_{\bar{R}}(\bar{H})|^{-|\bar{H}|\frac{p-1}{p}} |\bar{R}|^{|\bar{H}|\frac{p-1}{p}} \\
& = |C_{\bar{R}}(\bar{H})|^{-|\bar{H}|\frac{p-1}{p}} |\bar{R}|^{|\bar{H}|\frac{p-1}{p}}.
\end{split}
\end{equation*}

We now remove the bar notation using the definition of $R$ and (\ref{NNX CRX}),

\begin{equation*}
\begin{split}
\prod_{x \in \mathcal{T}} m_{Q\langle x \rangle} & = \left( \frac{|N_M(H)|}{|Q|} \right)^{-\frac{|H|}{|Q|}\frac{p-1}{p}}  \left( \frac{|N_M(Q)|}{|Q|} \right)^{\frac{|H|}{|Q|}\frac{p-1}{p}} \\
& = \left( \frac{|N_M(Q)|}{|N_M(H)|} \right)^{\frac{|H|}{|Q|}\frac{p-1}{p}}.
\end{split}
\end{equation*}

Using this bound in inequality (\ref{prod lambda outside H}) we get
\begin{equation*} 
\begin{split}
\prod_{x \in P \setminus H} \lambda_G(x) & \leq \left( \prod_{x \in \mathcal{T}} m_{Q\langle x \rangle} \right)^{|Q|} n_p(M)^{\frac{p-1}{p}|H|}  \\
& \leq \left( \frac{|N_M(Q)|}{|N_M(H)|} \right)^{\frac{p-1}{p}|H|} n_p(M)^{\frac{p-1}{p}|H|} \\
& \leq \left( \frac{|N_M(Q)|}{|N_M(H)|}  n_p(M)\right)^{\frac{p-1}{p}|H|},
\end{split}
\end{equation*}
and since $Q$ is a Sylow $p$-subgroup of $M$,
\begin{equation} \label{not in H}
\begin{split}
\prod_{x \in P \setminus H} \lambda_G(x)  \leq \left( \frac{|N_M(Q)|}{|N_M(H)|}  \frac{|M|}{|N_M(Q)|}\right)^{\frac{p-1}{p}|H|} = \left(  \frac{|M|}{|N_M(H)|}\right)^{\frac{p-1}{p}|H|}
\end{split}
\end{equation}

By combining (\ref{in H}) and (\ref{not in H}) we obtain
\begin{equation*}
\begin{split}
\prod_{x \in P} \lambda_G(x) & = \left( \prod_{x \in H} \lambda_G(x) \right)\left( \prod_{x \in P \setminus H} \lambda_G(x) \right) \\
& \leq \left( m_H^{\frac{p-1}{p}|H|} n_p(G)^\frac{|H|}{p} \right) \left(  \frac{|M|}{|N_M(H)|}\right)^{\frac{p-1}{p}|H|} \\
& = n_p(G)^\frac{|H|}{p}  \left(  m_H \frac{|M|}{|N_M(H)|}\right)^{\frac{p-1}{p}|H|}.
\end{split}
\end{equation*}
Finally, recalling (\ref{mX}) and the fact that $|H|=|P|/p$ we get
\begin{equation*}
\begin{split}
\prod_{x \in P} \lambda_G(x) & \leq n_p(G)^\frac{|H|}{p}  \left(  \frac{|N_M(H)|}{|N_M(P)|} \frac{|M|}{|N_M(H)|}\right)^{\frac{p-1}{p}|H|} \\
&= n_p(G)^\frac{|H|}{p} n_p(G)^{\frac{p-1}{p}|H|} = n_p(G)^\frac{|P|}{p}
\end{split}
\end{equation*}
which is against the fact that $G$ is a counterexample.
\end{proof}

\subsection*{Acknowledgements} 

This article is part of the author’s PhD thesis, which was written under the great supervision of Carlo Casolo, whose contribution to this work was essential. 

Thanks are also due to Francesco Fumagalli and Silvio Dolfi for his valuable comments and suggestions.

This work was partially funded by the Istituto Nazionale di Alta Matematica ``Francesco Severi" (Indam).

\end{document}